\documentclass{amsart}
\usepackage{mathtools}
\usepackage{amssymb}
\usepackage{etoolbox}
\usepackage[hidelinks]{hyperref}

\newcommand{\pr}[1]{\left(#1\right)}
\newcommand{\br}[1]{\left[#1\right]}
\DeclarePairedDelimiterX{\abs}[1]{\lvert}{\rvert}{\ifblank{#1}{\:\cdot\:}{#1}}
\DeclarePairedDelimiterX{\norm}[1]{\lVert}{\rVert}{\ifblank{#1}{\:\cdot\:}{#1}}
\DeclarePairedDelimiterX{\innerp}[2]{\langle}{\rangle}{{\ifblank{#1}{\:\cdot\:}{#1}}, {\ifblank{#2}{\:\cdot\:}{#2}}}

\providecommand{\st}{}
\newcommand{\SetSymbol}[1][]{%
	\nonscript\:\:#1\vert
	\allowbreak
	\nonscript\:\:
	\mathopen{}}
\DeclarePairedDelimiterX\set[1]\{\}{%
	\renewcommand\st{\SetSymbol[\delimsize]}#1}

\newcommand{\R}{\mathbb{R}}

\newcommand{\Z}{\mathbb{Z}}
\newcommand{\Q}{\mathbb{Q}}
\newcommand{\N}{\mathbb{N}}

\newcommand{\ep}{\varepsilon}
\newcommand{\fr}{\frac}
\newcommand{\dd}{\mathop{}\!\mathrm{d}}
\newcommand{\dnu}{\;\mathrm{d}\nu}
\newcommand{\andt}{\text{ and }}
\newcommand{\vchi}{\text{\large{$\chi$}}}
\newcommand{\Diff}{\mathrm{Diff}}
\newcommand{\what}{\widehat}
\newcommand{\fxpair}{\left(\left(f_k\right)_{k\in\N},x\right)}
\newcommand{\m}{\mathfrak{m}}
\newcommand{\sseq}{\subseteq}

\DeclareMathOperator{\supp}{supp}
\DeclareMathOperator{\diam}{diam}

\theoremstyle{plain}
\newtheorem{theorem}{Theorem}
\newtheorem{lemma}[theorem]{Lemma}
\newtheorem{corollary}[theorem]{Corollary}

\title[Exact dimensionality of stationary measures]{Exact dimensionality of stationary measures for random diffeomorphisms with all Lyapunov exponents equal and negative}
\author{Subhasish Mukherjee}

\address{Department of Mathematics, University of Chicago,
Chicago, IL 60637, USA}

\email{subhasish@uchicago.edu}
\date{\today}

\begin{document}
	
	\begin{abstract}
	We prove exact dimensionality of ergodic stationary measures for random compositions of $C^1$ diffeomorphisms of a closed manifold $M$ that contract at a single Lyapunov scale. More precisely, let $\nu$ be a Borel probability measure on $\Diff^1(M)$ such that $\log\sup_{x}\norm{D_xf}$ and $\log\sup_{x}\norm{D_xf^{-1}}$ are $\nu$-integrable, and let $\mu$ be an ergodic $\nu$-stationary probability measure. If the top and bottom Lyapunov exponents of the derivative cocycle coincide and equal $\lambda<0$, then $\mu$ is exact dimensional and $\dim(\mu)=h_\mu^{\mathrm{F}}(\nu)/(-\lambda)$, where $h_\mu^{\mathrm{F}}(\nu)$ is the Furstenberg entropy.
	\end{abstract}
	
	\maketitle

	\section{Introduction}

    A basic question in the dimension theory of dynamical systems is whether invariant and stationary measures have a well-defined dimension given by a Ledrappier--Young formula involving entropy and Lyapunov exponents. When looking at stationary measures of random compositions, most results of this kind assume either that the driving measure is discretely supported, the maps have algebraic structure, or both. We show that when our random dynamics contracts at a single asymptotic exponential rate, no such assumptions are needed: the driving measure may be any Borel probability measure on $\Diff^1(M)$ with a logarithmic moment condition, and the dimension is the Furstenberg entropy divided by the absolute value of the Lyapunov exponent. In particular, this gives the dimension formula for every ergodic, non-invariant stationary measure of a random walk on $\Diff^1(S^1)$ with the logarithmic moment condition.
    
	\subsection{Setup and statement} 
	Let $M$ be a closed smooth Riemannian manifold. We define $d: M \times M \to [0, \infty)$ to be the Riemannian distance between $x$ and $y$ when $x$ and $y$ are in the same connected component, and set $d(x, y) := D$ for $x$ and $y$ in different components, where $D$ is greater than the diameter of every component. We also set $d_M \coloneqq \dim(M) \ge 1$.  For a Borel probability measure $\mu$ on $M$
	and $x\in M$, define respectively  the \emph{lower} and \emph{upper pointwise dimensions of $\mu$ at $x$} as
	\[
	\underline d_\mu(x)
	\coloneqq \liminf_{r\to0}\frac{\log\mu(B(x,r))}{\log r}
	\quad\andt\quad
	\overline d_\mu(x)
	\coloneqq \limsup_{r\to0}\frac{\log\mu(B(x,r))}{\log r}.
	\]
	In general, we may have \smash[b]{$\underline{d}_\mu(x) \ne \overline d_\mu(x)$}, and both quantities may vary as $x$ varies. We say $\mu$ is \emph{exact dimensional} if there is $\delta\in[0,\infty]$
	with \smash{$\underline d_\mu(x)=\overline d_\mu(x)= \delta$ }for $\mu$-almost
	every $x$, and in this case we set $\dim(\mu)\coloneqq\delta$. Exact dimensionality tells us that the measure is regular enough so masses of balls scale in a coherent way across space. When $\mu$ is exact dimensional, $\dim(\mu)$ is also equal to the Hausdorff, packing, and several other notions of dimension of $\mu$ (see e.g. \cite[Section 4]{Young1982}, \cite[Chapter~10]{FalconerTechniques}).
	
	Dynamically defined measures are often exact dimensional because the masses
	and radii of dynamical neighborhoods are controlled by
	additive/subadditive quantities along orbits, so their asymptotic rates are
	governed by ergodic theorems. We study exact dimensionality of stationary measures on $M$ under random dynamics.
	
	To be more precise, let $\nu$ be a Borel probability measure on
	$G\coloneqq\Diff^1(M)$, equipped with its $C^1$ topology. We consider a random composition of i.i.d.\ diffeomorphisms drawn according to $\nu$ and call $\nu$ the \emph{driving measure} of the random dynamics.  We say a Borel probability measure $\mu$ on $M$ is
	\emph{$\nu$-stationary} if
	\[
	\nu*\mu\coloneqq\int_G f_*\mu\,\dnu(f)=\mu.
	\] Thus $\nu$-stationarity means that the distribution of a $\mu$-distributed
	point is preserved under one $\nu$-random diffeomorphism.
	A $\nu$-stationary measure $\mu$ is \emph{ergodic} if it is an extreme point of the convex set of all
	$\nu$-stationary measures.
	
	Let $\mu$ be any Borel probability measure on $M$. We pass to the product space recording both the driving sequence and the point in $M$. Set
	$\what X\coloneqq G^\N\times M$ with measure
	$\what\mu_\nu\coloneqq\nu^{\otimes\N}\times\mu$, and define the skew product map
	$\what F:\what X\to\what X$ by
	\[
	\what F\pr{\pr{f_k}_{k\in\N},x}
	=
	\pr{\pr{f_{k+1}}_{k\in\N},f_0(x)}.
	\]
	Then $\mu$ is $\nu$-stationary if and only if $\what\mu_\nu$ is $\what F$-invariant,
	and, assuming stationarity, $\mu$ is ergodic if and only if $\what\mu_\nu$ is
	$\what F$-ergodic (see e.g. \cite[Exercise 5.10]{Viana}). 
	
	We also record the exponential contraction rates of the random derivative cocycle. Let $\norm{A}$ and $\m(A)$ respectively denote the largest and smallest
	singular values of an invertible linear map $A$, so
	$\m(A)=\norm{A^{-1}}^{-1}$.  For $g\in G$, set
	$
	[g]_{C^1}\coloneqq\sup_{x\in M}\norm{D_xg},
	$ and suppose that 
    $
    f\mapsto \log[f]_{C^1}+\log[f^{-1}]_{C^1}$
	belongs to $L^1(\nu)$.  When $\mu$ is $\nu$-stationary ergodic, Kingman's subadditive ergodic theorem \cite{Kingman} shows there are
	$\lambda_{\mathrm{top}}, \lambda_{\mathrm{bot}} \in \R$ called respectively
	the \emph{top} and \emph{bottom Lyapunov exponents} of the derivative cocycle,
	such that for $\what\mu_\nu$-almost every
	\smash{$\pr{\pr{f_k}_{k\in\N},x}\in\what X$},
	\begin{align*}
		\lim_{n\to\infty}\frac1n
		\log\norm*{D_x(f_{n-1}\circ\cdots\circ f_0)}
		&=\lambda_{\mathrm{top}},\\
		\lim_{n\to\infty}\frac1n
		\log\m\pr{D_x(f_{n-1}\circ\cdots\circ f_0)}
		&=\lambda_{\mathrm{bot}}.
	\end{align*}
	Thus $\lambda_{\mathrm{top}}$ and $\lambda_{\mathrm{bot}}$ respectively dictate the fastest and slowest exponential rates of infinitesimal expansion along typical orbits.

	For Borel probability measures $\gamma_0,\gamma_1$ on $M$ with $\gamma_0\ll\gamma_1$, set
	\[
	D(\gamma_0\|\gamma_1)
	\coloneqq
	\int_M\log\frac{\dd\gamma_0}{\dd\gamma_1}\,\dd\gamma_0 = \int_M\frac{\dd\gamma_0}{\dd\gamma_1} 
	\log {\frac{\dd\gamma_0}{\dd\gamma_1}} \,\dd\gamma_1 \ge 0,
	\]
	where the last inequality follows since $t\log t\geq t-1$ for all $t \ge 0$.
	When $\gamma_0\not\ll\gamma_1$ we set $D(\gamma_0\|\gamma_1)=\infty$.
	We define the \emph{Furstenberg entropy} of $\mu$ with respect to $\nu$ as
	\[
	h_\mu^{\mathrm{F}}(\nu)
	\coloneqq
	\int_G D(f_*\mu\|\mu)\,\dnu(f) \in [0, \infty].
	\]
	For $\nu$-distributed $f$ and $\mu$-distributed $x$, $h_\mu^{\mathrm{F}}(\nu)$ measures how much information $f(x)$ reveals about $f$ itself.

	When the asymptotic dynamics are contracting at a single Lyapunov scale, our main theorem shows any ergodic stationary measure is exact dimensional, with dimension determined by the Furstenberg entropy and the Lyapunov exponent.

	\begin{theorem}\label{thm:main}
		Let $\nu$ be a Borel probability measure on $\Diff^1(M)$ such that $
		f\mapsto \log[f]_{C^1}+\log[f^{-1}]_{C^1}
		$
		belongs to $L^1(\nu)$. Suppose $\mu$ is a $\nu$-stationary ergodic Borel probability measure on $M$ and there is $\lambda<0$ such that
		$
		\lambda_{\mathrm{top}}
		=
		\lambda_{\mathrm{bot}}
		=
		\lambda.
		$
		Then $\mu$ is exact dimensional and
		\[
		\dim(\mu)=\frac{h_\mu^{\mathrm{F}}(\nu)}{-\lambda}.
		\]
	\end{theorem}

	In fact the proof gives the following one-sided bounds, which do not require conformality.

	\begin{theorem}\label{thm:bounds}
		Let $\nu$ be a Borel probability measure on $\Diff^1(M)$ such that $
		f\mapsto \log[f]_{C^1}+\log[f^{-1}]_{C^1}
		$
		belongs to $L^1(\nu)$ and suppose $\mu$ is a $\nu$-stationary ergodic Borel probability measure on $M$.
		\begin{enumerate}
			\item[(i)] If $\lambda_{\mathrm{top}}<0$, then $\overline d_\mu(x)\le h_\mu^{\mathrm F}(\nu)/(-\lambda_{\mathrm{top}})$ for $\mu$-almost every $x$.
            			\item[(ii)] If $\lambda_{\mathrm{bot}}<0$, then $\underline d_\mu(x)\ge h_\mu^{\mathrm F}(\nu)/(-\lambda_{\mathrm{bot}})$ for $\mu$-almost every $x$.
		\end{enumerate}
	\end{theorem}

	\subsection{Discussion}\label{disc}

    The dimension formula itself has a simple information-theoretic
	interpretation.  If $\mathbf f\sim\nu$ and $X\sim\mu$ are independent, then
	stationarity gives $\mathbf f(X)\sim\mu$ and
	$
	h_\mu^{\mathrm{F}}(\nu)$ is the mutual information $I(\mathbf f;\mathbf f(X)).
	$
	Thus Theorem \ref{thm:main}
	expresses dimension as the information revealed by one random step per
	unit of logarithmic spatial contraction.  
    
	The relation between entropy, Lyapunov exponents, and dimension goes back to
	the deterministic theory.  Young established the dimension formula for
	hyperbolic measures on surfaces \cite{Young1982}, and Ledrappier--Young
	prove the general dimension-entropy formula
	\cite{LedrappierYoung1985}.  Barreira--Pesin--Schmeling subsequently proved
	exact dimensionality for arbitrary hyperbolic ergodic invariant measures of
	$C^{1+\alpha}$ diffeomorphisms
	\cite{BarreiraPesinSchmeling1999}.  For random diffeomorphisms,
	Ledrappier--Young \cite{LedrappierYoung1988} and Liu--Xie
	\cite{LiuXie2006} obtained analogous results for sample measures, which are the fiberwise conditional measures associated to specific realizations of the random dynamics. In contrast, a stationary measure is the average over all such realizations. Exact dimensionality of the sample measures does not, in general, imply exact dimensionality of their stationary average.

	There are also direct results for stationary measures of nonlinear random dynamics. For a finite $C^1$ iterated function system of contractions and an ergodic shift-invariant measure $m$ on the coding space, Feng--Hu bound the lower and upper pointwise dimensions of the projected measure below and above by its projection entropy divided by the fastest and slowest exponential contraction rates, respectively \cite[Theorem~2.6]{FengHu2009}; when these two rates agree $m$-almost everywhere, they obtain exact dimensionality \cite[Theorem~2.8]{FengHu2009}. They also treat countable families of similarities that contract only on average, under a finite entropy assumption \cite[Theorem~10.1]{FengHu2009}. Their projection entropy agrees with the generalized Furstenberg entropy for the associated skew product, and in the Bernoulli case it reduces to the usual Furstenberg entropy.    Mihailescu--Urba\'nski extend the conformal case to random countable systems of uniform contractions over an ergodic base, under a finite entropy assumption \cite[Theorem~3.13 and Corollary~3.14]{MihailescuUrbanski2016}. Theorems~\ref{thm:main} and~\ref{thm:bounds} (and Appendix~\ref{app:general-base}) are analogues of these results for global diffeomorphisms that contract only on average, with an arbitrary driving law.
    
    On the circle, He--Jiao--Xu prove that, for a finitely supported $C^2$ random
	walk by orientation-preserving circle diffeomorphisms whose support preserves no common probability measure, every ergodic stationary measure is exact dimensional and satisfies the corresponding Furstenberg entropy--exponent formula \cite[Theorem~2.2]{HeJiaoXu2023}. Bezerra--Salcedo obtain the same conclusion for arbitrary driving measures on $\Diff^{1+\tau}(S^1)$ with a logarithmic moment condition, assuming that the semigroup generated by the support is proximal and has no common fixed point \cite[Theorem~5]{BezerraSalcedo2025}.
    
    More recently, in an unpublished manuscript, Brown--Rodriguez Hertz prove exact dimensionality for contracting stationary measures with a general negative Lyapunov spectrum and establish a Ledrappier--Young-type formula for Furstenberg entropy \cite[Corollary~2.7 and Theorem~2.9]{BrownRodriguezHertz2026}.  Their results assume
	$C^{1+\alpha}$ regularity, a discrete driving measure of finite Shannon entropy, and the corresponding logarithmic moment hypotheses \cite[\S2.5]{BrownRodriguezHertz2026}.
	
	Theorem~\ref{thm:main} treats a complementary regime.  The collapse to a single
	Lyapunov scale allows lower regularity and a more
	general driving law: the maps are only $C^1$, and $\nu$ may be an arbitrary
	Borel probability measure on $\Diff^1(M)$ satisfying the logarithmic $C^1$ moment assumption.  This also allows
	non-discrete laws of the sort arising, for example, from time-one maps in
	stochastic-flow models; see \cite{Baxendale1989} for related entropy and
	Lyapunov exponent theory for stochastic flows.     
    Theorem~\ref{thm:main} confirms the expectation in \cite[\S2.12.1]{BrownRodriguezHertz2026} that discreteness is unnecessary in the conformal setting.

    There are also analogous results when the maps have some algebraic structure. On the projective line, Ledrappier obtained the dimension formula with a slightly different notion of dimension without any discreteness assumption \cite[Part~V, Theorem~5.1 and Corollary~5.5]{Ledrappier}; Hochman--Solomyak prove exact dimensionality of the corresponding stationary measure \cite[Remark~1.2 and Theorem~3.4]{HochmanSolomyak2017}. Lessa proves exact dimensionality of the one-dimensional conditional measures of stationary measures on flag spaces, assuming uniqueness of the stationary measure and a similar logarithmic moment hypothesis \cite[Theorem~2]{Lessa}, which covers the projective line; Ledrappier--Lessa treat fibers of higher dimension with a single exponent, again without discreteness \cite[Theorem~2.7]{LL}, and use countable support only to pass to several exponents \cite[Theorem~2.8]{LL}, \cite{LedrappierLessa2023}; and Rapaport proves a Ledrappier--Young formula for finitely supported strongly irreducible and proximal random walks \cite{Rapaport2021}. For finite affine iterated function systems that contract on average with respect to an ergodic shift-invariant measure, Feng proves exact dimensionality and a Ledrappier--Young formula with no conformality assumption \cite[Theorems~1.2 and~1.3]{Feng2023}. Our proof adapts the strategy of Ledrappier, Lessa, and Ledrappier--Lessa.
    
The projective and flag
	structures in Ledrappier--Lessa also provide geometric filtrations adapted to
	distinct Lyapunov exponents. It is therefore natural to ask whether the
	$C^1$, arbitrary driving measure result proved here can be extended to
	several distinct negative Lyapunov exponents. For non-discrete driving measures this is open \cite[\S2.12.2]{BrownRodriguezHertz2026}. Theorem~\ref{thm:bounds} gives a partial answer: without conformality one still has the one-sided bounds, as in \cite[Theorem~2.6]{FengHu2009} and \cite[Lemma~8.1]{LL}. These are also exactly the bounds forced by a Ledrappier--Young formula $\dim(\mu)=\sum_i\gamma_i$, $h_\mu^{\mathrm F}(\nu)=\sum_i|\lambda_i|\gamma_i$ with partial dimensions $\gamma_i\ge0$, as established by Brown--Rodriguez Hertz under their hypotheses \cite[Corollary~2.7, Claim~2.8, and Corollary~4.16]{BrownRodriguezHertz2026}.

	Theorem~\ref{thm:bounds}(ii) is also closely related to the proof of Ledrappier's invariance principle \cite[Theorem~1]{LedrappierPositivity}: for a stationary sequence of matrices whose extreme Lyapunov exponents $\lambda_1\ge\lambda_d$ coincide, invariant measures on the projective bundle are determined by the future of the sequence; in particular, for i.i.d.\ products with $\lambda_1=\lambda_d$, every stationary measure on projective space is invariant under almost every matrix (see the proof of \cite[Corollary~1]{LedrappierPositivity} and \cite[Corollary~9.5]{Crauel90}; see also \cite[Theorem~B and Proposition~3.1]{AvilaViana2010} for smooth cocycles). That proof rests on the entropy inequality $\alpha\le(d-1)(\lambda_1-\lambda_d)$ \cite[Theorem~3]{LedrappierPositivity}, which Ledrappier deduces from a dimension estimate: for ergodic invariant measures, $\alpha\le\delta\cdot E\log(\norm{A}\norm{A^{-1}})$, where $\delta\le d-1$ is a dimension of the conditional measures defined in measure \cite[Proposition~4]{LedrappierPositivity}. Crauel extended the inequality to random dynamical systems on manifolds, obtaining  $h_\mu^{\mathrm F}(\nu)\le d_M\max\{0,-\lambda_{\mathrm{bot}}\}$ in our notation \cite[Theorem~6.1]{Crauel90}, \cite[Theorem~5.1]{Crauel93}. We use this bound in Lemma~\ref{furstconv} only to know that $h_\mu^{\mathrm F}(\nu)<\infty$ and $f_*\mu\ll\mu$. Theorem~\ref{thm:bounds}(ii) is a nonlinear analogue of Ledrappier's dimension estimate, with the asymptotic exponent $\lambda_{\mathrm{bot}}$ and the pointwise lower dimension; since $\underline d_\mu \le \overline d_\mu\le d_M$ almost everywhere by \eqref{eq:packing}, it sharpens Crauel's inequality when $\lambda_{\mathrm{bot}}<0$. The inequality also shows that $h_\mu^{\mathrm F}(\nu)>0$ forces $\lambda_{\mathrm{bot}}<0$ \cite[Corollary~5.4]{Crauel93}. In particular, when $d_M=1$, the conformality hypothesis is vacuous, and Theorem~\ref{thm:main} applies to every ergodic stationary measure of a random walk preserving no common probability measure on $\Diff^1(S^1)$ satisfying the moment condition. This recovers the exact dimensionality statements of \cite[Theorem~3.4]{HochmanSolomyak2017} (under a first moment assumption), \cite[Theorem~2.2]{HeJiaoXu2023}, and \cite[Theorem~5]{BezerraSalcedo2025}, in $C^1$ regularity and without proximality.

	Finally, Appendix~\ref{app:general-base} extends Theorems~\ref{thm:main} and~\ref{thm:bounds} to skew products over an arbitrary ergodic base, using the generalized Furstenberg entropy of \cite{BrownRodriguezHertz2026}. This covers stationary but not necessarily independent driving sequences, for which the conclusion concerns the fiber measures of the skew-product invariant measure. 

\section*{Acknowledgments}
Ideas in the proof are inspired by similar arguments given by Crauel, Ledrappier, Lessa, and Ledrappier--Lessa \cite{Crauel93, Ledrappier, Lessa,LedrappierLessa2023,LL}. The author also benefited greatly from discussions with Amie Wilkinson, James Marshall Reber, Pablo Lessa, and Aaron Brown.

\emph{Use of AI tools.} The author used Anthropic Claude Opus 5.5 in preparing this manuscript for: literature search and checking cited statements against the original sources; drafting parts of the exposition, in particular the discussion in Section~\ref{disc}; and checking the proofs, which led to corrections and simplifications in several places, particularly Lemmas~\ref{furstconv}, \ref{lem:maker-average}, and \ref{lem:general-entropy}.
The main results and the strategy of proof are the author's own. The author has checked all mathematical content and references and takes full responsibility for them.
 
\emph{Data availability.} Data sharing is not applicable to this article.

	\section{Proof}
	
	\subsection{Proof outline} The proof given here is largely self-contained, starting with Crauel's estimate \cite{Crauel93} which implies $f_\ast \mu \ll \mu$ for $\nu$-almost every $f$, and then using some basic facts from ergodic theory and geometric measure theory. Throughout the proof, we assume the hypotheses in the statement of Theorem \ref{thm:bounds} and fix the corresponding $\nu$, $\mu$, $\lambda_{\rm{top}}$, and $\lambda_{\rm{bot}}$. We also fix a Besicovitch number $N_B\ge1$ such that the Besicovitch covering theorem holds in $M$, i.e. every family of balls has a countable subfamily that covers the set of their centers and has multiplicity at most $N_B$ (see for example \cite[2.8.9 and 2.8.14]{Federer}).
	
	Fixing $\ep>0$, we first replace $\nu$ by
	$
	\eta=\frac12\nu^{*N}+\frac12\nu^{*(N+1)}
	$
	for $N$ sufficiently large depending on $\ep$.  This gives
	$h_\mu^{\mathrm{F}}(\eta)=(N+\frac12)h_\mu^{\mathrm{F}}(\nu)$, preserves ergodicity of $\mu$, and makes the expected upper and
	lower logarithmic derivative scales arbitrarily close to
	$(N+\frac12)\lambda_{\rm{top}}$ and $(N+\frac12)\lambda_{\rm{bot}}$ respectively.
	
	We then pass to the invertible natural extension and consider radii
	$r_{n,j}^\pm$ that follow the contraction of balls along backward
	orbits as $j$ ranges from $0$ to $n$.  The resulting ball inclusions give telescoping inequalities
	expressing the change in logarithmic ball mass as a sum of local
	log-quotients.  The terminal radii $r_{n,n}^\pm$ satisfy
	$\frac1n\log r_{n,n}^\pm\longrightarrow (N+\tfrac12)\pr{\lambda+O(\ep)},
	$
	whereas the initial radii $r_{n,0}^\pm$ have subexponential scale and contribute
	no asymptotic ball mass.
	A Maker's ergodic theorem argument then identifies the averages of these local log-quotients with the Furstenberg entropy
	$h_\mu^{\mathrm{F}}(\eta)$.  Putting it together, we have
	$
	h_\mu^{\mathrm{F}}(\nu)/(-\lambda_{\rm{bot}}+2\ep )
	\le \underline d_\mu(x)
	\le \overline d_\mu(x)
	\le {h_\mu^{\mathrm{F}}(\nu)}/(-\lambda_{\rm{top}}-2\ep)
	$
	for $\mu$-almost every $x$, and letting $\ep\to0$ proves the theorem.
    
	\subsection{Change of driving measure}
    
    We now proceed to the details of the proof. In Lemmas \ref{furstconv} and \ref{ergchar}, we assume that $\kappa$ is a Borel probability measure on $G$ such that  $\mu$ is $\kappa$-stationary and $f \mapsto \log[f]_{C^1} + \log[f^{-1}]_{C^1}$ is in $L^1(\kappa)$. Recall that the convolution $\kappa^{\ast n}$ is defined as the pushforward of $\kappa^{\otimes n}$ under $(f_{0}, f_1, \dots, f_{n - 1}) \mapsto f_{n - 1} \circ f_{n - 2} \circ \cdots \circ f_0$. Note that convolution is associative, and so $\kappa^{\ast n} \ast \mu  = \mu$. In addition, $\log[f]_{C^1}\ge0$ for every $f\in G$, and $[g\circ f]_{C^1}\le[g]_{C^1}[f]_{C^1}$, so it follows that $f\mapsto\log[f]_{C^1}+\log[f^{-1}]_{C^1}$ is nonnegative and belongs to $L^1(\kappa^{*n})$ for every $n\ge1$.
    
    We first examine how Furstenberg entropy behaves under convolution.
	\begin{lemma}\label{furstconv}
		Suppose that $\mu$ is also $\kappa$-ergodic. We have $h_\mu^{\mathrm F}(\kappa) < \infty$, and so $f_\ast \mu \ll \mu$ for $\kappa$-almost every $f \in G$, and $h_\mu^{\mathrm{F}}(\kappa^{\ast n}) = n h_\mu^{\mathrm{F}}(\kappa)$.
	\end{lemma}
	\begin{proof}
				By Crauel
		\cite[Corollary~5.3]{Crauel93}, we have
		$h_\mu^{\mathrm{F}}(\kappa)\le d_M\max\{0, -\lambda_{\mathrm{bot}}(\kappa)\}<\infty$, where $\lambda_{\mathrm{bot}}(\kappa)$ is the bottom Lyapunov exponent for $(\kappa,\mu)$. Thus
		$D(f_*\mu\|\mu)<\infty$ for $\kappa$-almost every $f\in G$, and so 
		$f_*\mu\ll\mu$ and $\log\frac{\dd f_*\mu}{\dd\mu}\in L^1(f_*\mu)$ for all such $f$. Furthermore, $\mu$ is $\kappa^{\ast n}$ stationary as discussed.
        
        We show inductively that $h_\mu^{\mathrm{F}}(\kappa^{*n})=nh_\mu^{\mathrm{F}}(\kappa)$. This is clear when $n = 1$, so suppose  \smash{$h_\mu^{\mathrm{F}}(\kappa^{*n})=nh_\mu^{\mathrm{F}}(\kappa)$
		for some $n\ge1$. Then $h_\mu^{\mathrm{F}}(\kappa^{*n})<\infty$, so
		$f_*\mu\ll\mu$} for $\kappa^{*n}$-almost every $f$. Thus for
		$(\kappa^{*n}\otimes\kappa)$-almost every $(f,g)$, we have
		$
		g_*(f_*\mu)\ll g_*\mu\ll\mu,
		$ and so
		the Radon-Nikodym chain rule gives
		\[
		\frac{\dd(g\circ f)_*\mu}{\dd\mu}(g(f(x)))
		=
		\frac{\dd f_*\mu}{\dd\mu}(f(x))
		\frac{\dd g_*\mu}{\dd\mu}(g(f(x)))
		\]
		for $\mu$-almost every $x$. 
		All terms are integrable by finiteness of
		$h_\mu^{\mathrm{F}}(\kappa^{*n})$ and $h_\mu^{\mathrm{F}}(\kappa)$, so by Fubini we have as desired
		\begin{align*}
			h_\mu^{\mathrm{F}}\left(\kappa^{*(n+1)}\right)
			&=
			\int
			\!\log\frac{\dd f_*\mu}{\dd\mu}(f(x))\,
			\dd\kappa^{*n}(f)\dd\kappa(g)\dd\mu(x)\\
			&\quad+
			\int
			\!\log\frac{\dd g_*\mu}{\dd\mu}(g(f(x)))\,
			\dd\kappa^{*n}(f)\dd\kappa(g)\dd\mu(x)\\
			&=h_\mu^{\mathrm{F}}(\kappa^{*n})
			+\int
			\log\frac{\dd g_*\mu}{\dd\mu}(g(y))\,
			\dd(\kappa^{*n}*\mu)(y)\dd\kappa(g)\\
			&=h_\mu^{\mathrm{F}}(\kappa^{*n})+h_\mu^{\mathrm{F}}(\kappa) =(n+1)h_\mu^{\mathrm{F}}(\kappa). \qedhere
		\end{align*}
	\end{proof}

	We recall that the \emph{Markov operator} $P_\kappa: L^2(\mu) \to L^2(\mu)$ is defined by $P_\kappa u(x):=\int_Gu(f(x))\,\dd\kappa(f)$. For $u \in L^2(\mu)$, Jensen's inequality gives
	\[
	\norm{P_\kappa u}_{L^2(\mu)}^2
	\le
	\int_{G\times M}|u(f(x))|^2\,\dd\kappa(f)\dd\mu(x)
	=
	\norm{u}_{L^2(\mu)}^2,
	\]
	where the last equality follows from $\kappa$-stationarity.
	In other words, $P_\kappa$ is a contraction on $L^2$. The next lemma records a useful characterization of stationary ergodic measures using the Markov operator. For a proof, see \cite[Proposition 5.13]{Viana}.
	\begin{lemma}\label{ergchar}
		$\mu$ is $\kappa$-ergodic if and only if the only $P_\kappa$ invariant functions in $L^2(\mu)$ are constant $\mu$-almost everywhere.
	\end{lemma}
	
	We now use the previous two lemmas to replace $\nu$ with a different driving measure better suited to our estimates. The idea is similar to that in \cite[Section 8.1]{LL} while using a little extra averaging to recover ergodicity.
	
	\begin{lemma}\label{etalem}
		Let $\ep>0$. Then there is $N\ge1$ such that, setting
		\[
		\eta:=\frac12\nu^{*N}+\frac12\nu^{*(N+1)},
		\]
		$\mu$ is $\eta$-stationary and $\eta$-ergodic, $f_\ast \mu \ll \mu$ for $\eta$-a.e. $f$,
		$
		h_\mu^{\mathrm{F}}(\eta)=\pr{N+\fr12}h_\mu^{\mathrm{F}}(\nu),
		$
		and
		\begin{align*}
			\abs*{\int_{G\times M}\log\m(D_xf)\,\dd\eta(f)\dd\mu(x)-\pr{N+\fr12}\lambda_{\mathrm{bot}}}
			&<\pr{N+\fr12}\ep, \andt\\
			\abs*{\int_{G\times M}\log\norm{D_xf}\,\dd\eta(f)\dd\mu(x)-\pr{N+\fr12}\lambda_{\mathrm{top}}}
			&<\pr{N+\fr12}\ep.
		\end{align*}
	\end{lemma}
	
	\begin{proof}
		For $\fxpair \in \what X$ and $n \ge 1$, set
		\[
		c_n^\pm\fxpair \coloneqq \log\left\|\pr{D_x(f_{n-1}\circ\cdots\circ f_0)}^{\pm 1}\right\|.
		\]
		Note then that $(c_n^\pm)_{n \in \N}$ are both subadditive sequences, and 
		\[
		-\sum_{j = 0}^{n - 1} \log \br{f_j^{\mp 1}}_{C^1} \le c_n^\pm\fxpair \le \sum_{j = 0}^{n - 1} \log \br{f_j^{\pm 1}}_{C^1},
		\]
		so $c_n^\pm \in L^1(\what \mu_\nu)$ for all $n \in \N$.
		Kingman's subadditive ergodic theorem then implies
		$
		\frac1n\int_{\what X}c_n^+ \;\dd\what\mu_\nu\to\lambda_{\mathrm{top}}$ and $\frac1n\int_{\what X}c_n^- \;\dd\what\mu_\nu\to-\lambda_{\mathrm{bot}}.
		$
		In particular, we have by the definition of $\nu^{\ast n}$ that
		\begin{align*}
			&\fr1n \int_{G\times M}\log\norm*{D_xf}\;\dd\nu^{\ast n}(f)\;\dd\mu(x)\\
			=\;  &\fr1n \begin{multlined}[t]
				\int_{G^n\times M}\log\norm*{D_x\pr{f_{n - 1} \circ f_{n - 2} \circ \cdots \circ f_0}}
				\dd\nu^{\otimes n}(f_0, \dots, f_{n - 1})\;\dd\mu(x)
			\end{multlined}\\
			=\;  &\frac1n\int_{\what X}c_n^+ \;\dd\what\mu_\nu\longrightarrow \lambda_{\mathrm{top}}.
		\end{align*}
		Similarly, 
		\[
		\fr 1n \int_{G\times M} \log\mathfrak{m}\pr{ D_x f }\;\dd\nu^{\ast n}(f)\;\dd\mu(x) \longrightarrow \lambda_{\mathrm{bot}}.
		\]
		Choose $N$ large enough such that the two normalized integrals above differ from $\lambda_{\mathrm{top}}$ and $\lambda_{\mathrm{bot}}$, respectively, by less than $\ep$ for $n=N,N+1$. Averaging these bounds then gives the desired inequalities. Since $h_\mu^{\mathrm{F}}(\kappa)$ is affine in $\kappa$ and $h_\mu^{\mathrm{F}}(\nu^{\ast n}) = nh_\mu^{\mathrm{F}}(\nu)$ by Lemma \ref{furstconv}, we also get $h_{\mu}^{\mathrm{F}}(\eta)= (N + \fr12) h_{\mu}^{\mathrm{F}}(\nu)<\infty$; in particular $f_*\mu\ll\mu$ for $\eta$-almost every $f$.
		
		It remains to prove ergodicity. Since $P_\eta=\tfrac12(P_\nu^N+P_\nu^{N+1})$, every $u\in L^2(\mu)$ satisfying $P_\eta u=u$ also satisfies
		$
		\norm{u}_2 = \norm{P_\eta u}_2\le\tfrac12\pr{\norm{P_\nu^Nu}_2+\norm{P_\nu^{N+1}u}_2}\le\norm{u}_2
		$ since $P_\nu$ is a contraction. Since we have equality, strict convexity of $L^2(\mu)$ gives $P_\nu^Nu=P_\nu^{N+1}u=u$, so $P_\nu u=u$. Lemma~\ref{ergchar} implies that $u$ is constant, and another application of Lemma~\ref{ergchar} shows that $\mu$ is $\eta$-ergodic.
	\end{proof}

	The next lemma estimates radii of images of balls under one step of the $\eta$-walk.
	
		\begin{lemma}\label{lem:ball-scales}
		Let $\ep$ and $\eta$ be as in Lemma~\ref{etalem}. There is $r_0>0$ and there are Borel functions $a^\pm:G\times M\to(0,\infty)$ such that, for every $f\in G$, $x\in M$, and $r\in(0,r_0]$,
		\[
		B(f(x),a^-(f,x)r)\subset f[B(x,r)]\subset B(f(x),a^+(f,x)r).
		\]
		Moreover, $\log a^\pm\in L^1(\eta\otimes\mu)$, and, setting
		$
		A^\pm:=\int_{G\times M}\log a^\pm(f,x)\,\dd\eta(f)\dd\mu(x),
		$
		we have
		\[
		\abs*{A^--(N+\tfrac12)\lambda_{\mathrm{bot}}}<2(N+\tfrac12)\ep
		\quad\andt\quad
		\abs*{A^+-(N+\tfrac12)\lambda_{\mathrm{top}}}<2(N+\tfrac12)\ep.
		\]
	\end{lemma}
	
	\begin{proof}
		As noted above, $f\mapsto\log[f]_{C^1}+\log[f^{-1}]_{C^1}$ is nonnegative and $\eta$-integrable. Since $G$ is a Polish space, the Borel probability measure $\eta$ is tight, so there is a compact $K\subseteq G$ such that
		\[
		\int_{K^c}\pr{\log[f]_{C^1}+\log[f^{-1}]_{C^1}}\,\dd\eta(f)
		<\pr{N+\fr12}\frac\ep2.
		\]
		Since $(f,x)\mapsto\log\norm{D_xf}$ is uniformly continuous on $K\times M$, there is $r_1>0$ such that for all $f\in K$ and $x,y\in M$ with $d(x,y)<r_1$, we have
		$
		\norm{D_yf}\le e^{\pr{N+\fr12}\ep/2}\norm{D_xf}.
		$
		The mean value theorem then implies
		\[
		d(f(x),f(y))\le e^{\pr{N+\fr12}\ep/2}\norm{D_xf}\,d(x,y)
		< e^{\pr{N+\fr12}\ep/2}\norm{D_xf} r
		\]
		for all $r\le r_1$, $f\in K$, and $x,y\in M$ with $d(x,y)<r$.
		
		Since inversion is continuous on $G$, the set $K^{-1}:=\{f^{-1}:f\in K\}$ is also compact, so similarly there is $r_2>0$ such that for all $f\in K$ and $x,y\in M$ with $d(x,y)<r_2$, we have
		$
		d(f^{-1}(x),f^{-1}(y))\le e^{\pr{N+\fr12}\ep/2}\norm{D_xf^{-1}}\,d(x,y).
		$ In particular, if
		$
		r\le r_2/\sup_{f\in K}[f]_{C^1}
		$
		and
		\[
		d(f(x),y)<e^{-\pr{N+\fr12}\ep/2}\m(D_xf)r,
		\]
		then $d(f(x),y)<r_2$, hence 
		\[
		d(x,f^{-1}(y))
		\le e^{\pr{N+\fr12}\ep/2}
		\norm{D_{f(x)}f^{-1}}d(f(x),y)
		<r.
		\]
		With this in mind, we define $r_0:=\min\{r_1,r_2/\sup_{f\in K}[f]_{C^1}\}$. Then for all $f\in K$, $x\in M$, and $r\in(0,r_0]$, we have
		\[
		B\pr{f(x),e^{-\pr{N+\fr12}\ep/2}\m(D_xf)r}
		\subset f[B(x,r)]
		\subset B\pr{f(x),e^{\pr{N+\fr12}\ep/2}\norm{D_xf}r}.
		\]
		For general $f\in G$, $x\in M$, and $r>0$, we also have the elementary containments
		\[
		B\pr{f(x),\br{f^{-1}}_{C^1}^{-1}r}
		\subset f[B(x,r)]
		\subset B\pr{f(x),[f]_{C^1}r}.
		\]
		Thus define
		\begin{align*}
			a^-(f,x)&:=
			\begin{cases}
				e^{-\pr{N+\fr12}\ep/2}\m(D_xf),&f\in K,\\
				[f^{-1}]_{C^1}^{-1},&f\notin K,
			\end{cases}\quad \andt\\
			a^+(f,x)&:=
			\begin{cases}
				e^{\pr{N+\fr12}\ep/2}\norm{D_xf},&f\in K,\\
				[f]_{C^1},&f\notin K.
			\end{cases}
		\end{align*}
		Note $a^\pm$ are Borel, the required inclusions hold, and $a^-\le a^+$. Moreover, our initial analysis implies both differences
		\[
		\log\m(D_xf)-\log a^-(f,x) \quad \andt \quad  \log a^+(f,x)-\log\norm{D_xf}
		\]
		are nonnegative and bounded above by
		\[
		\pr{N+\fr12}\frac\ep2\vchi_K(f)
		+\pr{\log[f]_{C^1}+\log[f^{-1}]_{C^1}}\vchi_{K^c}(f).
		\]
		
		Thus $\log a^\pm\in L^1(\eta\otimes\mu)$ and we have
		\begin{align*}
			0&\le\int_{G\times M}\log\m(D_xf)\,\dd\eta(f)\dd\mu(x)-A^-<\pr{N+\fr12}\ep, \andt\\
			0&\le A^+-\int_{G\times M}\log\norm{D_xf}\,\dd\eta(f)\dd\mu(x)<\pr{N+\fr12}\ep.
		\end{align*}
		Combining the above bounds with Lemma~\ref{etalem} gives the desired result.
	\end{proof}

	\subsection{Two-sided extension and backwards telescope} \label{twoside}
	
	To compare the masses of balls and apply ergodic arguments along backward iterates, we now pass to the two-sided natural extension. Fix $\ep$ and the corresponding $N,\eta,r_0,a^\pm$, and $A^\pm$ as in Lemmas \ref{etalem}  and \ref{lem:ball-scales}.
	Let $X:=G^\Z\times M$ and define $F: X \to X$ by
	\[
	F\pr{(f_k)_{k\in\Z},x}:=\pr{(f_{k+1})_{k\in\Z},f_0(x)}.
	\]
	Let $\pi_M$, $\pi_{G^\Z}$, and $\pi_{G^\N\times M}$ be the corresponding projection maps from $X$. Define $\widehat\mu_\eta:=\eta^{\otimes\N}\otimes\mu$. We let $\mu_\eta$ be the natural extension measure of $\what \mu_\eta$, i.e. $\mu_\eta$ is the unique $F$-invariant measure on $X$ such that
	\[
	\pr{\pi_{G^\Z}}_\ast \mu_\eta = \eta^{\otimes\Z}
	\quad\text{and}\quad
	\pr{\pi_{G^\N\times M}}_\ast \mu_\eta = \widehat\mu_\eta.
	\]
	Then $(X,\mu_\eta,F)$ is ergodic (see e.g. \cite[Chapter 5.4]{Viana} for more details).

	For $z = \pr{\pr{f_k}_{k\in\Z},x}\in X$, write $x_{n}\coloneqq \pi_M F^n(z)$, so $x_0 = x$ and $x_{n + 1} = f_n(x_n)$. It follows that for all Borel $\varphi:G\times M\to \R$ such that either side below converges absolutely,
	\begin{equation}\label{eq:integral}
		\int_X\varphi(f_{-1},x_0)\;\dd\mu_\eta(z)
		=\int_{G\times M}\varphi(f,f(x))\,\dd\mu(x)\dd\eta(f).
	\end{equation}

	Set $X_0:=\bigcap_{k\in\Z}\{z\in X:x_k\in\supp\mu\}$, and note $X_0$ is Borel, $F$-invariant, and $\mu_\eta(X_0)=1$. 
	For $(z,r)\in X\times(0,r_0]$, define
	\[
	\psi_r(z):=
	\begin{cases}\displaystyle\log\frac{\mu(f_{-1}^{-1}B(x_0,r))}{\mu(B(x_0,r))},&z\in X_0,\\[2mm]
		0,&z\notin X_0.
	\end{cases}
	\]
	Note the map $(z,r)\mapsto\psi_r(z)$ is Borel, since
	$
	(f,x,r)\mapsto\mu(f^{-1}B(x,r))
	$
	is Borel. We then set $\psi(z) \coloneqq \lim_{n \to \infty} \psi_{1/n}(z)$ wherever this limit exists, and $\psi(z) \coloneqq 0$ otherwise. Thus $\psi$ is also Borel in $z$.

    By Lemma \ref{furstconv} we have $f_*\mu\ll\mu$ and $D(f_*\mu\|\mu)<\infty$ for
	$\eta$-almost every $f\in G$. For all such $f$ and for $\mu$-almost every $x$, the Besicovitch differentiation theorem implies
	\[
	\lim_{r \to 0} 
	\frac{(f_*\mu)(B(x,r))}{\mu(B(x,r))} = \frac{\dd f_*\mu}{\dd\mu}(x).
	\]
	In particular,
	$
	\psi(z) = \log\frac{\dd(f_{-1})_*\mu}{\dd\mu}(x_0)
	$
	for $\mu_\eta$-almost every $z$.
	By \eqref{eq:integral}, we also see
	\[
	\int \psi(z) \,\dd\mu_\eta(z) = h_\mu^{\mathrm{F}}(\eta).
	\]
	
	With these preliminaries out of the way, we now pick our good sequence of radii along the orbit. For $n\ge1$, $0\le j\le n$, and $z\in X$, define
	\[
	r_{n,j}^\pm(z):=r_0\exp\left(
	\sum_{k=-n}^{-n+j-1}\log a^\pm(f_k,x_k)
	-\max_{0\le i\le n}\sum_{k=-n}^{-n+i-1}\log a^\pm(f_k,x_k)
	\right).
	\]
	It then follows that for $0 \le j \le n$ we have $0<r_{n,j}^\pm\le r_0$ and 
	\begin{equation}\label{eq:radii}
		r_{n,j}^\pm=a^\pm(f_{-n+j-1},x_{-n+j-1})r_{n,j-1}^\pm.
	\end{equation}
	
	The following lemma gives us elementary inequalities that naturally split our estimates into three separate terms we can bound individually; the lemma and the radii above are similar to the constructions in \cite[Part~V, Theorem~5.1]{Ledrappier}, \cite[Proposition~7]{LedrappierPositivity}, and \cite[Sections~8.3--8.4]{LL}.
	
	\begin{lemma}\label{lem:telescope}
		For $\mu_\eta$-almost every $z\in X$ and every $n\ge1$,
		\begin{align*}
			\log\mu\pr{B(x_0,r_{n,n}^-)}
			&\le\log\mu\left(B(x_{-n},r_{n,0}^-)\right) -\sum_{j=1}^n\psi_{r_{n,j}^-}(F^{-n+j}z), \andt\\
			\log\mu\left(B(x_0,r_{n,n}^+)\right)
			&\ge\log\mu\left(B(x_{-n},r_{n,0}^+)\right) -\sum_{j=1}^n\psi_{r_{n,j}^+}(F^{-n+j}z).
		\end{align*}
	\end{lemma}
	
	\begin{proof}
		Let $z \in X_0$. For $1\le j\le n$, by Lemma \ref{lem:ball-scales} and \eqref{eq:radii} we have
		\[
		f_{-n+j-1}^{-1}B(x_{-n+j},r_{n,j}^-)
		\subset B(x_{-n+j-1},r_{n,j-1}^-),
		\]
		so
		\[
		\exp\pr{\psi_{r_{n,j}^-}(F^{-n+j}z)}\mu\pr{B(x_{-n+j},r_{n,j}^-)}
		\le\mu\pr{B(x_{-n+j-1},r_{n,j-1}^-)}.
		\]
		Multiplying these inequalities over $1\le j\le n$ and taking logarithms gives the first desired inequality. The second desired inequality is analogous.
	\end{proof}
	
	We also have the following estimates on the growth rate of the radii which will be essential in the final proof.

	\begin{lemma}\label{lem:maker}
		If $A^\pm < 0$, then for $\mu_\eta$-almost every $z\in X$ and for every $\theta\in(0,1)$, we have as $n \to \infty$
		\[ \frac1n\log r_{n,0}^\pm(z)\to0, \qquad
		\frac1n\log r_{n,n}^\pm(z)\to A^\pm,
		\quad \andt \quad \max_{\lceil\theta n\rceil\le j\le n}r_{n,j}^\pm(z)\to 0.
		\]
	\end{lemma}
	
	\begin{proof}
		For $m\ge0$, set
		$
		S_m^\pm(z):=\sum_{k=-m}^{-1}\log a^\pm(f_k,x_k).
		$
		Birkhoff's ergodic theorem applied to $F^{-1}$ gives $S_m^\pm(z)=mA^\pm+o(m)$
		for $\mu_\eta$-almost every $z$. In particular,
		\[
		\max_{0\le m\le n}\left|S_m^\pm-mA^\pm\right|=o(n),
		\] 
		so uniformly in $j$ we have
		$
		\sum_{k=-n}^{-n+j-1}\log a^\pm(f_k,x_k)
		=
		S_n^\pm-S_{n-j}^\pm
		=
		jA^\pm+o(n).
		$
		Since $A^\pm<0$, we get
		\[
		\max_{0\le j\le n}
		\sum_{k=-n}^{-n+j-1}\log a^\pm(f_k,x_k) = \max\set*{0, \max_{1\le j\le n}
			\sum_{k=-n}^{-n+j-1}\log a^\pm(f_k,x_k)}
		=o(n).
		\]
		Thus for all $0 \le j \le n$, we have \[
		\log r_{n,j}^\pm
		=
		\log r_0+jA^\pm+o(n).\]
		Taking $j=n$ and $j=0$ gives respectively
		$
		\frac1n\log r_{n,n}^\pm\to A^\pm$ and $\frac1n\log r_{n,0}^\pm\to0$.
		Also, for $\theta\in(0,1)$ and $\lceil\theta n\rceil\le j\le n$, we have
		$
		\log r_{n,j}^\pm
		\le
		\log r_0+\theta nA^\pm+o(n) \to -\infty
		$
		 since $A^\pm<0$. Therefore, we have
		$
		\max_{\lceil\theta n\rceil\le j\le n}r_{n,j}^\pm \to 0
		$ as desired.
	\end{proof}
	
	\subsection{Asymptotic estimates}

	We will use two elementary estimates for computing dimension using balls of varying radii as in the terms appearing in Lemma \ref{lem:telescope}.
	\begin{lemma}\label{lem:endpoint-mass}
	Let $(Y,\gamma)$ be a probability space. For every $n\ge1$, let
	$g_n:Y\to M$ and $s_n:Y\to(0,r_0]$ be measurable, and suppose
	$(g_n)_*\gamma=\mu$. If $\frac1n\log s_n(y)\to0$ for
	$\gamma$-almost every $y\in Y$, then
	\[
	\frac1n\log\mu(B(g_n(y),s_n(y)))\longrightarrow0
	\]
	for $\gamma$-almost every $y\in Y$.
\end{lemma}
\begin{proof}
	Recall $d_M=\dim M$. By relative volume comparison
	\cite[Lemma~7.1.4]{Petersen} and compactness of $M$, there is $C_0>0$
	such that every $r/2$-separated subset of $M$ has cardinality at most
	$C_0r^{-d_M}$ for $r\in(0,r_0]$. Hence, if $\set{x_i}_{i}$ is a
	maximal $r/2$-separated subset of
	$\set*{x\in M\st\mu(B(x,r))<q}$, then
	\[
	\set*{x\in M\st\mu(B(x,r))<q}\sseq\bigcup_i B(x_i,r/2).
	\]
	We thus get the packing inequality
	\begin{equation}\label{eq:packing}
		\mu\pr{\set*{x\in M\st\mu(B(x,r))<q}}
		\le \sum_i\mu(B(x_i,r/2))
		\le C_0qr^{-d_M}.
	\end{equation}
	
	Fix an integer $m\ge1$. For each $n\ge1$, set
	$\rho_n:=e^{-n/(2d_Mm)}$ and
	\[
	E_n:=
	\set*{
		y\in Y\st
		s_n(y)\ge\rho_n,\
		\mu(B(g_n(y),s_n(y)))<e^{-n/m}
	}.
	\]
	If $y\in E_n$, then $\rho_n\le s_n(y)$, so
	$\mu(B(g_n(y),\rho_n))<e^{-n/m}$. Thus
	\[
	E_n\sseq
	g_n^{-1}\pr{
		\set*{x\in M\st\mu(B(x,\rho_n))<e^{-n/m}}}.
	\]
	Since $(g_n)_*\gamma=\mu$ and $\rho_n\le r_0$ for all sufficiently
	large $n$, it follows using \eqref{eq:packing} that
	\[
	\gamma(E_n)
	\le C_0e^{-n/m}\rho_n^{-d_M}
	= C_0e^{-n/(2m)}
	\]
	for all such $n$. Therefore $\sum_{n\ge1}\gamma(E_n)<\infty$,
	and so by Borel--Cantelli, $y\notin E_n$ for large enough $n$
	and $\gamma$-almost every $y$.
	
	Furthermore, $\log s_n(y)/n\to0$ almost everywhere, so
	$s_n(y)\ge\rho_n$ for large enough $n$.
	Consequently
	$
	\mu(B(g_n(y),s_n(y)))\ge e^{-n/m}
	$
	for large enough $n$, so
	\[
	\liminf_{n\to\infty}
	\frac1n\log\mu(B(g_n(y),s_n(y)))\ge-\frac1m.
	\]
	Intersecting the corresponding conull sets over $m\in\N$ and
	letting $m\to\infty$ gives the desired lower bound.
	Since $\log\mu(B(g_n(y),s_n(y)))\le0$, the reverse bound is
	immediate and we have the desired result.
\end{proof}

	\begin{lemma}\label{lem:interpolation}
		Let $x\in\supp\mu$ and for $n \in \N$ let $s_n>0$ and $t < 0$ with
		$\log s_n/n\to t$ as $n \to \infty$. Then we have
		\[
		\limsup_{n\to\infty}\frac{\log\mu(B(x,s_n))}{n}
		= t{\underline d_\mu(x)}, \quad \andt
		\quad
		\liminf_{n\to\infty}\frac{\log\mu(B(x,s_n))}{n}
		=
		t{\overline d_\mu(x)}.
		\]
	\end{lemma}
	\begin{proof}
		We prove the equality for $\underline d_\mu(x)$; the proof of the other equality is analogous. Since $s_n \to 0$ and $1/t < 0$, we see
		\begin{align*}
			\underline d_\mu(x)
			&\le\liminf_{n\to\infty}
			\frac{\log\mu(B(x,s_n))}{\log s_n}\\
			&=\liminf_{n\to\infty}
			\frac{n}{\log s_n}
			\frac{\log\mu(B(x,s_n))}{n}=\fr1t\limsup_{n\to\infty}
			\frac{\log\mu(B(x,s_n))}{n},
		\end{align*}
		which gives one inequality. For the other direction, first note if
		$\underline d_\mu(x)=\infty$ then the preceding inequality gives the
		result, so suppose otherwise. For $k\in\N$ choose $r_k>0$ with
		$r_k\to0$ and
		$
		{\log\mu(B(x,r_k))}/{\log r_k}\to\underline d_\mu(x).
		$
		Fix $\delta\in(0,1)$ and set
		$
		n_k:=
		\left\lfloor
		(1-\delta){\log r_k}/{t}
		\right\rfloor.
		$
		Then $n_k\to\infty$ and
		$$
		\frac{\log s_{n_k}}{\log r_k} = {\frac{\log s_{n_k}}{n_k}}{\frac{n_k}{\log r_k}}
		\to 1-\delta < 1.
		$$
		Since $\log r_k<0$, it follows
		$s_{n_k}>r_k$ for all sufficiently large $k$. In particular,
		\begin{align*}
			\limsup_{n\to\infty}
			\frac{\log\mu(B(x,s_n))}{n}
			&\ge\limsup_{k\to\infty}
			\frac{\log\mu(B(x,r_k))}{n_k} \\
			&=\lim_{k\to\infty}
			\fr{\log r_k}{n_k}
			\frac{\log\mu(B(x,r_k))}{\log r_k}=\frac t{1-\delta}\,\underline d_\mu(x).
		\end{align*}
		Since this holds for all $\delta \in (0, 1)$, we obtain as desired
		\[
		\limsup_{n\to\infty}
		\frac{\log\mu(B(x,s_n))}{n}
		=
		t\,\underline d_\mu(x). \qedhere
		\]
	\end{proof}

	In order to estimate Furstenberg entropy, we will approximate Radon-Nikodym derivatives by quotients of ball measures, and use a Maker's theorem-type argument. The next lemma verifies the necessary integrability hypotheses to do so. The proof is a modification of
	\cite[Lemma~8.3]{LL} and \cite[Lemma~9]{Lessa} from their projective setting.
	
	\begin{lemma}\label{lem:ball-quotient}
		There is $C<\infty$, depending only on $M$, such that
		whenever $\gamma_0\ll\gamma_1$ are Borel probability measures on $M$
		with $D(\gamma_0 \| \gamma_1) <\infty$, we have
		$
		\frac{\gamma_0(B(x,r))}{\gamma_1(B(x,r))}
		\to
		\frac{\dd\gamma_0}{\dd\gamma_1}(x)
		$
		for $\gamma_0$-almost every $x$ as $r\to0$, and
		\[
		\int_M
		\sup_{r > 0}
		\abs*{\log\frac{\gamma_0(B(x,r))}{\gamma_1(B(x,r))}}
		\,\dd\gamma_0(x)
		\le C\pr{1+D(\gamma_0 \| \gamma_1)}.
		\]
	\end{lemma}
	\begin{proof}
		The first assertion is the Besicovitch differentiation theorem for
		finite Borel measures  (see e.g. Federer \cite[2.8.14--2.8.15]{Federer}). For the maximal inequality, we first note for $\gamma_0$-almost every $x$ that $\gamma_0(B(x, r))$ and $\gamma_1(B(x, r))$ are both positive for all $r>0$. Moreover, the suprema below may be restricted to
		$r\in\Q_{>0}$, since $\gamma_i(B(x,q))\to\gamma_i(B(x,r))$ as
		$q\to r$ from below. In particular, the functions
		\[
		W^-(x):=\sup_{r>0}\log
		\frac{\gamma_1(B(x,r))}{\gamma_0(B(x,r))} \;\;\andt\;\;
		W^+(x):=\sup_{r>0}\log
		\frac{\gamma_0(B(x,r))}{\gamma_1(B(x,r))}
		\]
		are Borel and defined $\gamma_0$-almost everywhere. Since $B(x,r)=M$ for $r>\diam M$, we also have
		$W^\pm\ge0$. It suffices to show that $\int W^\pm \,\dd\gamma_0 \le C\pr{1+D(\gamma_0 \| \gamma_1)}$.
		
		Fix $t\ge0$
		and set $E_t:=\set*{x \in M\st W^-(x)>t}$. For each $x\in E_t$, pick $s_x>0$ such
		that
		$
		\gamma_0(B(x,s_x))
		<
		e^{-t}\gamma_1(B(x,s_x)).
		$ Recall $N_B$ is the Besicovitch covering constant of $M$. 
		By the Besicovitch covering theorem, there is a countable subfamily
		$\{B_i\} \sseq \{B(x, s_x)\}$ covering $E_t$ with $\sum_i \vchi_{B_i} \leqq N_B$, so we see
		\[
		\gamma_0(E_t)
		\le \sum_i\gamma_0(B_i)
		<e^{-t}\sum_i\gamma_1(B_i)
		\le N_Be^{-t}.
		\]
		The layer-cake formula then gives
		$
		\int_MW^-\,\dd\gamma_0
		=
		\int_0^\infty\gamma_0\pr{E_t}\,\dd t
		\le
		N_B.
		$

		It remains to bound $\int_M W^+ \,\dd\gamma_0$. Set $J=\dd\gamma_0/\dd\gamma_1$ and define
		\[
		M_{\gamma_1}J(x)
		\coloneqq
		\sup_{r>0}
		\frac{1}{\gamma_1(B(x,r))}
		\int_{B(x,r)}J\,\dd\gamma_1,
		\]
		with the convention $0/0=0$. Fix $s\ge1$ and set
		$E_s=\{M_{\gamma_1}J>s\}$. For each $x\in E_s$, choose
		$r_x>0$ such that
		$
		\int_{B(x,r_x)}J\,\dd\gamma_1
		>
		s\,\gamma_1(B(x,r_x)).
		$
		Since $J\vchi_{\{J\le s/2\}}\le s/2$, we see
		\[
		\int_{B(x,r_x)}
		J\vchi_{\{J>s/2\}}\,\dd\gamma_1
		>
		\frac{s}{2}\gamma_1(B(x,r_x)).
		\]
		By the Besicovitch covering theorem, there is a countable subfamily
		$\{B_j\}$ of the balls $\{B(x,r_x)\}$ that covers $E_s$ and has
		multiplicity at most $N_B$, and so
		\begin{equation}\label{eq:maximal-tail}
			\gamma_1(E_s)
			\le \sum_j\gamma_1(B_j) <
			\frac{2}{s}\sum_j
			\int_{B_j}J\vchi_{\{J>s/2\}}\,\dd\gamma_1 \le
			\frac{2N_B}{s}
			\int_{\{J>s/2\}}J\,\dd\gamma_1 .
		\end{equation}
		
		Since $B(x,r)=M$ for $r>\diam M$, we have
		$M_{\gamma_1}J \ge 1$. Thus applying the layer-cake theorem, \eqref{eq:maximal-tail}, and
		Tonelli's theorem, we obtain
		\[
		\begin{aligned}
			\int_M M_{\gamma_1}J\,\dd\gamma_1
			&\le
			1+2N_B\int_1^\infty
			\frac1s\int_{\{J>s/2\}}J\,\dd\gamma_1\,\dd s \\
			&=
			1+2N_B\int_M
			J\log(2J)\vchi_{\{J>1/2\}}\,\dd\gamma_1.
		\end{aligned}
		\]
		Since $-t\log t\le e^{-1}$ on $[0,1]$ and
		$\int_MJ\,\dd\gamma_1=1$, we can further bound
		\begin{align*}
			&\int_{M} J \log(2J)\vchi_{\set*{J > 1/2}} \;\dd\gamma_1\\ =\; 
			&\int_{M} J \log J \;\dd\gamma_1
			- \int_{M} J \log J \vchi_{\set*{J \le 1/2}} \;\dd\gamma_1+ \log 2 \int_{M} J \vchi_{\set*{J > 1/2}}\;\dd\gamma_1 \\
			\le\; &{ D(\gamma_0 \| \gamma_1)+ e^{-1} + \log 2 }.
		\end{align*}
		Consequently,
		\begin{equation}\label{eq:maximal-L1}
			\int_M M_{\gamma_1}J\,\dd\gamma_1
			\le
			1+2N_B\bigl(D(\gamma_0\Vert\gamma_1)+e^{-1}+\log2\bigr).
		\end{equation}
		
		We note $\gamma_0(B(x,r))/\gamma_1(B(x,r))\le M_{\gamma_1}J(x)$ whenever $\gamma_1(B(x,r))>0$, so we have $W^+\le\log M_{\gamma_1}J$. Since $J>0$
		$\gamma_0$-a.e., Jensen's inequality and \eqref{eq:maximal-L1} give
		\begin{align*}
			\int_M W^+\,\dd\gamma_0
			&\le \int \log J \,\dd\gamma_0 + \int_M \log \frac{M_{\gamma_1}J}{J}\,\dd\gamma_0
			\\
			&\le
			D(\gamma_0\Vert\gamma_1)
			+\log \int_M\frac{M_{\gamma_1}J}{J}\,\dd\gamma_0 \\
			&\le
			D(\gamma_0\Vert\gamma_1)
			+\log\int_M M_{\gamma_1}J\,\dd\gamma_1 \le
			C\left(1+D(\gamma_0\Vert\gamma_1)\right). 
		\end{align*}
		for some $C > 0$ depending only on $N_B$ and hence only on $M$, which shows the desired result.
	\end{proof}
	
	\begin{corollary}\label{cor:psi}
		Using the notation in Section \ref{twoside}, there is some $\Psi \in L^1(\mu_\eta)$ such that $|\psi_r(z)|+|\psi(z)|\le \Psi(z)$ for all $0<r\le r_0$.
	\end{corollary}
	
	\begin{proof}
		For $z \in X$, set $$\Psi(z):=2\sup_{q\in\Q\cap(0,r_0]}|\psi_q(z)|.$$ Note $\Psi$ is Borel as a countable supremum of Borel functions. For every $r\in(0,r_0]$, choose
		$q_j\in\Q\cap(0,r]$ with $q_j \to r$ monotonically. Continuity from below of the measure gives
		$\psi_{q_j}(z)\to\psi_r(z)$ whenever $z \in X_0$, and so
		$
		|\psi_r(z)|\le \Psi(z)/2.
		$
		Letting $r\to0$ also shows $|\psi(z)|\le \Psi(z)/2$, so
		$|\psi_r|+|\psi|\le \Psi$.

		Thus using Lemma~\ref{lem:ball-quotient} and \eqref{eq:integral}, we have as desired
		\begin{align*}
			\int_X \Psi \;\dd\mu_\eta &= 2\int_{G \times M} \sup_{q\in\Q\cap(0,r_0]} \abs*{\log\frac{\mu(f^{-1}B(x,q))}{\mu(B(x,q))}} \;\dd \pr{f_\ast\mu}(x) \dd\eta(f)\\
			&\le
			2C\int_G\left(1+D(f_*\mu\|\mu)\right)\dd\eta(f)
			=
			2C\pr{1+h_\mu^{\mathrm{F}}(\eta)}<\infty. \qedhere
		\end{align*}
	\end{proof}

	As mentioned, we have the following Maker's theorem style argument as in \cite{Maker}.
	
	\begin{lemma}\label{lem:maker-average}
		If $A^\pm < 0$, then for $\mu_\eta$-almost every $z\in X$, we have
		\[
		\frac1n\sum_{j=1}^n
		\psi_{r_{n,j}^\pm}(F^{-n+j}z)
		\longrightarrow h_\mu^{\mathrm{F}}(\eta).
		\]
	\end{lemma}
	
	\begin{proof}
		Let $\Psi$ be as in Corollary \ref{cor:psi}.
		For $0<\delta\le r_0$, set
		$$
		\Delta_\delta(z):=
		\sup_{q\in\Q\cap(0,\delta]}|\psi_q(z)-\psi(z)|.
		$$
		Since $\psi_r \to \psi$, $\Delta_\delta(z) \to 0$, $\mu_\eta$-almost everywhere, and
		$0\le\Delta_\delta(z) \le\Psi(z)$ for $\mu_\eta$-almost every $z$. By the dominated convergence
		theorem,
		$
		\int_X\Delta_\delta\,\dd\mu_\eta\to0
		$
		as $\delta\to 0$. Continuity from below of the measure also
		gives $|\psi_r-\psi|\le\Delta_\delta$ whenever $0<r\le\delta$.
		
		Fix $\theta\in(0,1)$. By Lemma~\ref{lem:maker}, for almost every
		$z$ and all sufficiently large $n$,
		$r_{n,j}^\pm(z)\le\delta$ whenever
		$\lceil\theta n\rceil\le j\le n$. Thus
		\begin{align*}
			\frac1n\sum_{j=1}^n
			\left|
			\psi_{r_{n,j}^\pm}(F^{-n+j}z)
			-\psi(F^{-n+j}z)
			\right|
			&\le \frac1n\sum_{0 \le j<\lceil\theta n\rceil}
			\Psi(F^{-n+j}z)\\
			&\quad+\frac1n\sum_{j=\lceil\theta n\rceil}^n
			\Delta_\delta(F^{-n+j}z).
		\end{align*}
		
		Taking $\limsup$ and applying Birkhoff's ergodic theorem, we get
		\begin{align*}
			&\limsup_{n\to\infty}\frac1n\sum_{j=1}^n
			\left|
			\psi_{r_{n,j}^\pm}(F^{-n+j}z)
			-\psi(F^{-n+j}z)
			\right|\le \theta\int_X\Psi\,\dd\mu_\eta+
			\int_X\Delta_\delta\,\dd\mu_\eta,
		\end{align*}
		 for $\mu_\eta$-almost every $z$.
		Letting first $\delta\to 0$ and then $\theta\to 0$ shows
		that the left hand side above goes to $0$ as $n \to \infty$. So we see again by Birkhoff's theorem that
		\begin{align*}
			\lim_{n\to\infty}\frac1n\sum_{j=1}^n
			\psi_{r_{n,j}^\pm}(F^{-n+j}z)
			= \lim_{n\to\infty}\frac1n\sum_{j=1}^n\psi(F^{-n+j}z)=\int_X\psi\,\dd\mu_\eta=h_\mu^{\mathrm{F}}(\eta),
		\end{align*}
		for $\mu_\eta$-almost every $z$ as desired.
	\end{proof}
	
		\subsection{Proof of Theorem \ref{thm:main}}\label{sec:proof} We again write $z = \pr{(f_k)_{k \in \Z}, x}$ and use the same notation as in Section \ref{twoside}. Assume that $\lambda_{\mathrm{top}}=\lambda_{\mathrm{bot}}=\lambda<0$ and that $\ep\in(0,-\lambda/4)$. Then $A^\pm<0$ by Lemma~\ref{lem:ball-scales}, so Lemmas~\ref{lem:maker} and~\ref{lem:maker-average} apply to both signs. For every $n\ge1$, note $\pr{\pi_M \circ F^{-n}}_*\mu_\eta=\mu$ and $\fr1n{\log r_{n,0}^\pm} \to 0$ by Lemma~\ref{lem:maker}, and so applying Lemma~\ref{lem:endpoint-mass} with $g_n = \pi_M \circ F^{-n}$ and $s_n=r_{n,0}^\pm$, we have
	\[
	\fr1n\log\mu\pr{B\pr{\pi_M \circ F^{-n}(z),r_{n,0}^\pm}} = \fr1n\log\mu\pr{B\pr{x_{-n},r_{n,0}^\pm}}\to0
	\]
	for $\mu_\eta$-almost every $z$. We also see $\fr1n \log r_{n,n}^\pm\to A^\pm<0$ by Lemma~\ref{lem:maker}, so using Lemma~\ref{lem:interpolation} with $s_n = r_{n,n}^\pm$ gives
	\begin{align*}
		A^-{\underline d_\mu(x_0)}
		&=\limsup_{n\to\infty}\frac{\log\mu\pr{B\pr{x_0,r_{n,n}^-}}}{n}, \andt\\
		A^+{\overline d_\mu(x_0)}
		&=\liminf_{n\to\infty}\frac{\log\mu\pr{B\pr{x_0,r_{n,n}^+}}}{n},
	\end{align*}
	for $\mu_\eta$-almost every $z$.
	Dividing by $n$, taking $\limsup$/$\liminf$ in Lemma~\ref{lem:telescope}, and applying Lemma~\ref{lem:maker-average}  therefore gives
	\[
	{A^-}{\underline d_\mu(x_0)}\le -h_\mu^{\mathrm{F}}(\eta)
	\quad \andt \quad
	{A^+}{\overline d_\mu(x_0)}\ge -h_\mu^{\mathrm{F}}(\eta),
	\]
	for $\mu_\eta$-almost every $z$. The functions $\underline d_\mu$ and $\overline d_\mu$ are Borel, and $(\pi_M)_*\mu_\eta=\mu$, so the preceding inequalities hold for $\mu$-almost every $x\in M$. Lemma~\ref{lem:ball-scales} and the fact that $h_\mu^{\mathrm{F}}(\eta) = (N + \fr12)h_\mu^{\mathrm{F}}(\nu)$ shows
	\[
	\frac{h_\mu^{\mathrm{F}}(\nu)}{-\lambda+2\ep}
	\le\underline d_\mu(x)
	\le\overline d_\mu(x)
	\le\frac{h_\mu^{\mathrm{F}}(\nu)}{-\lambda-2\ep},
	\]
	for $\mu$-almost every $x$. Applying this to a sequence $\ep_m\to0$ and intersecting the corresponding conull sets, we obtain
	\[
	\underline d_\mu(x)=\overline d_\mu(x)=\frac{h_\mu^{\mathrm{F}}(\nu)}{-\lambda},
	\]
	for $\mu$-almost every $x$, as desired. Theorem~\ref{thm:bounds} follows in the same way, using only the sign $+$ when $\lambda_{\mathrm{top}}<0$ (with $\ep\in(0,-\lambda_{\mathrm{top}}/4)$) and only the sign $-$ when $\lambda_{\mathrm{bot}}<0$ (with $\ep\in(0,-\lambda_{\mathrm{bot}}/4)$). \qed

\appendix

\section{Skew products over an ergodic base}
\label{app:general-base}

We extend Theorems~\ref{thm:main} and~\ref{thm:bounds} to skew products over an arbitrary ergodic base, in the notation of \cite[\S4.1]{BrownRodriguezHertz2026}. Lemma~\ref{lem:general-entropy} replaces Lemma~\ref{furstconv}, a coin flip between $N$ and $N+1$ in Lemma~\ref{lem:blocking} replaces the convolution $\eta$ of Lemma~\ref{etalem}, and the rest of Sections~\ref{twoside}--\ref{sec:proof} carry over with minor changes.

\subsection{Setting and statement}
Let $M$ be as before, let $(\widehat\Omega,\widehat{\mathcal B},\widehat\nu)$ be a standard probability space, and let $\widehat\theta:\widehat\Omega\to\widehat\Omega$ be an ergodic $\widehat\nu$-preserving transformation. Let $\omega\mapsto f_\omega\in G$ be measurable, set $\widehat X:=\widehat\Omega\times M$, and define $\widehat F(\omega,x):=(\widehat\theta\omega,f_\omega(x))$. Let $\widehat\mu$ be an ergodic $\widehat F$-invariant Borel probability measure on $\widehat X$ whose projection to $\widehat\Omega$ is $\widehat\nu$, and fix a measurable disintegration
\[
\widehat\mu=\int_{\widehat\Omega}
\delta_\omega\otimes\widehat\mu_\omega\,\dd\widehat\nu(\omega).
\]
Assume that $\omega\mapsto\log[f_\omega]_{C^1}+\log[f_\omega^{-1}]_{C^1}$ belongs to $L^1(\widehat\nu)$. For $n\ge1$, write $f_\omega^{(n)}:=f_{\widehat\theta^{n-1}\omega}\circ\cdots\circ f_\omega$, so that $\widehat F^n(\omega,x)=(\widehat\theta^n\omega,f_\omega^{(n)}(x))$. As in the introduction, Kingman's theorem gives constants $\lambda_{\mathrm{top}}$ and $\lambda_{\mathrm{bot}}$ equal, respectively, to  $$\lim_{n \to \infty}\frac1n\log\norm{D_xf_\omega^{(n)}} \andt \lim_{n \to\infty}\frac1n\log\m(D_xf_\omega^{(n)})$$ $\widehat\mu$-almost everywhere. Following \cite[\S3.9]{BrownRodriguezHertz2026}, the \emph{Furstenberg entropy} of $\widehat\mu$ is defined as $h^{\mathrm F}(\widehat\mu):=\int_{\widehat\Omega}D\pr{(f_\omega)_*\widehat\mu_\omega\|\widehat\mu_{\widehat\theta\omega}}\,\dd\widehat\nu(\omega)$.

In the i.i.d.\ setting of Theorem~\ref{thm:main}, namely $\widehat\Omega=G^{\N}$, $\widehat\nu=\nu^{\otimes\N}$, $\widehat\theta$ the shift, $f_\omega=\omega_0$, and $\widehat\mu=\what\mu_\nu$, we have $\widehat\mu_\omega=\mu$ for $\widehat\nu$-almost every $\omega$ and $h^{\mathrm F}(\widehat\mu)=h_\mu^{\mathrm F}(\nu)$, so the following theorem implies Theorems~\ref{thm:main} and~\ref{thm:bounds}. 

\begin{theorem}\label{thm:general-base}
For $\widehat\nu$-almost every $\omega$ and $\widehat\mu_\omega$-almost every $x$, the following hold.
\begin{enumerate}
\item[(i)] If $\lambda_{\mathrm{top}}<0$, then $\overline d_{\widehat\mu_\omega}(x)\le h^{\mathrm F}(\widehat\mu)/(-\lambda_{\mathrm{top}})$.
\item[(ii)] If $\lambda_{\mathrm{bot}}<0$, then $\underline d_{\widehat\mu_\omega}(x)\ge h^{\mathrm F}(\widehat\mu)/(-\lambda_{\mathrm{bot}})$.
\end{enumerate}
In particular, if $\lambda_{\mathrm{top}}=\lambda_{\mathrm{bot}}=\lambda<0$, then for $\widehat\nu$-almost every $\omega$ the measure $\widehat\mu_\omega$ is exact dimensional and $\dim(\widehat\mu_\omega)=h^{\mathrm F}(\widehat\mu)/(-\lambda)$.
\end{theorem}

We now proceed to the proof as before, highlighting where any major changes might be necessary.

\begin{lemma}\label{lem:general-entropy}
We have $h^{\mathrm F}(\widehat\mu)\le d_M\max\{0,-\lambda_{\mathrm{bot}}\}<\infty$. Consequently, for $\widehat\nu$-almost every $\omega$ we have $(f_\omega)_*\widehat\mu_\omega\ll\widehat\mu_{\widehat\theta\omega}$, and for every $n\ge1$,
\begin{equation}\label{eq:general-entropy-iterate}
h^{\mathrm F}(\widehat\mu,n):=
\int_{\widehat\Omega}
D\pr{(f_\omega^{(n)})_*\widehat\mu_\omega
\big\|\widehat\mu_{\widehat\theta^n\omega}}
\,\dd\widehat\nu(\omega)
=nh^{\mathrm F}(\widehat\mu).
\end{equation}
\end{lemma}

\begin{proof}
Let $\tau:(\Omega,\mathcal B_\Omega,\nu_\Omega,\theta)\to(\widehat\Omega,\widehat{\mathcal B},\widehat\nu,\widehat\theta)$ be the natural extension of the base. Define $F:\Omega\times M\to\Omega\times M$ by
\[
F(\xi,x):=\pr{\theta\xi,f_{\tau\xi}(x)}.
\]
We note $F$ is invertible and $(\tau\times\mathrm{id}_M)\circ F=\widehat F\circ(\tau\times\mathrm{id}_M)$. Let $\widetilde\mu$ be the natural extension measure of $\widehat\mu$, i.e.\ $\widetilde\mu$ is the unique $F$-invariant measure on $\Omega\times M$ such that $(\tau\times\mathrm{id}_M)_*\widetilde\mu=\widehat\mu$. Then $(\Omega\times M,\widetilde\mu,F)$ is ergodic, and its Lyapunov exponents are the same as those of $\widehat\mu$. Restricting $F$ and $\widetilde\mu$ to $\tau^{-1}\widehat{\mathcal B}\otimes\mathcal B(M)$ gives a one-sided random dynamical system with an invariant measure in the sense of \cite[Section~2]{Crauel93}, whose fiber measures are $\xi\mapsto\widehat\mu_{\tau\xi}$; by our moment assumption and the submultiplicativity of $[\,\cdot\,]_{C^1}$, this system satisfies the hypotheses of \cite[Theorem~5.1]{Crauel93} (compare \cite[Remark~(i) after Corollary~5.3]{Crauel93}). Then as before, \cite[Corollary~5.3]{Crauel93} readily implies the desired bound on $h^{\mathrm F}(\widehat\mu)$ and \cite[Corollary~3.5]{Crauel93} gives \eqref{eq:general-entropy-iterate}.
\end{proof}

Fix $\ep>0$. For every $n\ge1$, we note similar to before that
\begin{equation}\label{eq:general-iterate-bound}
\log[(f_\omega^{(n)})^{\pm1}]_{C^1}
\le\sum_{j=0}^{n-1}
\log[f_{\widehat\theta^j\omega}^{\pm1}]_{C^1},
\end{equation}
so the Kingman argument in the proof of Lemma~\ref{etalem} applies verbatim and gives
\[
\frac1n\int_{\widehat X}\log\m(D_xf_\omega^{(n)})
\,\dd\widehat\mu\longrightarrow\lambda_{\mathrm{bot}}
\quad\andt\quad
\frac1n\int_{\widehat X}\log\norm{D_xf_\omega^{(n)}}
\,\dd\widehat\mu\longrightarrow\lambda_{\mathrm{top}}.
\]
Choose $N$ so that the two normalized integrals differ from $\lambda_{\mathrm{bot}}$ and $\lambda_{\mathrm{top}}$, respectively, by less than $\ep$ for $n=N$ and $n=N+1$.

The map $\widehat F^N$ need not be ergodic, so we use an auxiliary coin flip between $\widehat F^N$ and $\widehat F^{N + 1}$ (compare the auxiliary construction in \cite[\S11.2]{EskinLindenstrauss}). Set $\Sigma:=\{N,N+1\}^{\N}$ and $\beta:=(\tfrac12\delta_N+\tfrac12\delta_{N+1})^{\otimes\N}$, write $\zeta=(\zeta_j)_{j\in\N}\in\Sigma$, and let $\sigma$ be the shift on $\Sigma$. Set $\widehat\Omega_N:=\Sigma\times\widehat\Omega$. Define the maps $\widehat\theta_N: \what \Omega_N \to \what \Omega_N$ and $\widehat F_N: \widehat\Omega_N\times M \to \widehat\Omega_N\times M$ by
\[
\widehat\theta_N(\zeta,\omega)
:=(\sigma\zeta,\widehat\theta^{\zeta_0}\omega)
\quad\andt\quad
\widehat F_N(\zeta,\omega,x)
:=\pr{\widehat\theta_N(\zeta,\omega),
f_\omega^{(\zeta_0)}(x)},
\]
and set $\widehat\nu_N:=\beta\otimes\widehat\nu$ and $\widehat\mu_N:=\beta\otimes\widehat\mu$. Thus $\widehat F_N$ is a skew product with fiber maps $f_\omega^{(\zeta_0)}$, and $\widehat\mu_N$ has fiber measures $\widehat\mu_\omega$ over $(\zeta,\omega)$; in the i.i.d.\ case, $\widehat F_N$ is an extension of the skew product of the $\eta$-random dynamics defined in Lemma~\ref{etalem}. The corresponding Furstenberg entropy is then \[h^{\mathrm F}(\widehat\mu_N):=\int_{\widehat\Omega_N}D\pr{(f_\omega^{(\zeta_0)})_*\widehat\mu_\omega\|\widehat\mu_{\widehat\theta^{\zeta_0}\omega}}\,\dd\widehat\nu_N(\zeta,\omega).\]

\begin{lemma}\label{lem:blocking}
The measure $\widehat\mu_N$ is $\widehat F_N$-invariant and ergodic, its Furstenberg entropy is $h^{\mathrm F}(\widehat\mu_N)=(N+\tfrac12)h^{\mathrm F}(\widehat\mu)$, and we have the bounds
\begin{equation}\label{eq:general-blocking}
\begin{aligned}
\abs*{\int\log\m(D_xf_\omega^{(\zeta_0)})\,\dd\widehat\mu_N(\zeta,\omega,x)-\pr{N+\fr12}\lambda_{\mathrm{bot}}}&<\pr{N+\fr12}\ep, \andt\\
\abs*{\int\log\norm{D_xf_\omega^{(\zeta_0)}}\,\dd\widehat\mu_N(\zeta,\omega,x)-\pr{N+\fr12}\lambda_{\mathrm{top}}}&<\pr{N+\fr12}\ep.
\end{aligned}
\end{equation}
\end{lemma}

\begin{proof}
Since $\zeta_0$ and $\sigma\zeta$ are independent, and $\sigma_*\beta=\beta$, we note
\[
(\widehat F_N)_*\widehat\mu_N
=\beta\otimes\frac12
\pr{(\widehat F^N)_*\widehat\mu
+(\widehat F^{N+1})_*\widehat\mu}
=\widehat\mu_N.
\]
The integrands in \eqref{eq:general-blocking} depend on $\zeta$ only through $\zeta_0$, and the sets $\set*{\zeta\in\Sigma\st\zeta_0=N}$ and $\set*{\zeta\in\Sigma\st\zeta_0=N+1}$ each have $\beta$-measure $\frac12$, so we see
\[
\int\log\m(D_xf_\omega^{(\zeta_0)})\,\dd\widehat\mu_N(\zeta,\omega,x)
=\frac12\sum_{n\in\{N,N+1\}}\int_{\widehat X}\log\m(D_xf_\omega^{(n)})\,\dd\widehat\mu(\omega,x),
\]
and similarly with $\norm{\cdot}$ in place of $\m$. By the choice of $N$, for $n\in\{N,N+1\}$ we have
\begin{align*}
\abs*{\int_{\widehat X}\log\m(D_xf_\omega^{(n)})\,\dd\widehat\mu-n\lambda_{\mathrm{bot}}}&<n\ep,\\
\abs*{\int_{\widehat X}\log\norm{D_xf_\omega^{(n)}}\,\dd\widehat\mu-n\lambda_{\mathrm{top}}}&<n\ep,
\end{align*}
so averaging over $n=N,N+1$ gives us \eqref{eq:general-blocking}. In addition, \eqref{eq:general-entropy-iterate} readily implies
\[
h^{\mathrm F}(\widehat\mu_N)=\tfrac12\pr{h^{\mathrm F}(\widehat\mu,N)+h^{\mathrm F}(\widehat\mu,N+1)}=\pr{N+\tfrac12}h^{\mathrm F}(\widehat\mu).
\]

Finally, note that $\widehat F_N$ is the skew product over $(\Sigma,\beta,\sigma)$ of the random dynamics on $\widehat X$ with driving measure $\frac12\delta_{\widehat F^N}+\frac12\delta_{\widehat F^{N+1}}$, for which $\widehat\mu$ is stationary. Lemma~\ref{ergchar} also applies with $M$ replaced by $\widehat X$, so the exact same argument as in Lemma~\ref{etalem} shows the desired ergodicity.
\end{proof}

\subsection{Proof of Theorem~\ref{thm:general-base}}
Fix $\ep>0$, and let $N$, $\widehat F_N$, and $\widehat\mu_N$ be as in Lemma~\ref{lem:blocking}. We follow Sections~\ref{twoside}--\ref{sec:proof} and indicate the necessary changes.

Let $\eta_N$ be the pushforward of $\widehat\nu_N$ under $(\zeta,\omega)\mapsto f_\omega^{(\zeta_0)}$. Then the proof of Lemma~\ref{lem:ball-scales} applies with $\eta$ replaced by $\eta_N$, and with integrals of functions of $(f,x)$ against $\eta\otimes\mu$ replaced by integrals of functions of $(f_\omega^{(\zeta_0)},x)$ against $\widehat\mu_N$. Together with Lemma \ref{lem:blocking}, this gives $r_0>0$ and Borel functions $a^\pm$ as in Lemma~\ref{lem:ball-scales}, with $\log a^\pm(f_\omega^{(\zeta_0)},x)\in L^1(\widehat\mu_N)$. The integrals $A^\pm:=\int\log a^\pm(f_\omega^{(\zeta_0)},x)\,\dd\widehat\mu_N$ then satisfy
\begin{equation}\label{eq:general-scales}
\abs*{A^--(N+\tfrac12)\lambda_{\mathrm{bot}}}<2(N+\tfrac12)\ep
\quad\andt\quad
\abs*{A^+-(N+\tfrac12)\lambda_{\mathrm{top}}}<2(N+\tfrac12)\ep.
\end{equation}

Let $(X_N,\mu_N,F_N)$ be the ergodic natural extension of $(\widehat\Omega_N\times M,\widehat\mu_N,\widehat F_N)$, with factor map $\pi:X_N\to\widehat\Omega_N\times M$. For $z\in X_N$ and $k\in\Z$, we write $$\pi F_N^k(z)=\pr{\zeta^{(k)},\omega_k,x_k} \andt f_k:=f_{\omega_k}^{(\zeta^{(k)}_0)},$$ so that $\omega_{k+1}=\widehat\theta^{\zeta^{(k)}_0}\omega_k$ and $x_{k+1}=f_k(x_k)$. Since $(\pi\circ F_N^{-1})_*\mu_N=\widehat\mu_N$ and $x_0=f_{-1}(x_{-1})$, we have, in place of \eqref{eq:integral}, the equality
\begin{equation}\label{eq:general-integral}
\int_{X_N}\!\!\varphi(\zeta^{(-1)},\omega_{-1},x_0)
\dd\mu_N(z)
=
\int_{\widehat\Omega_N}\!\int_M\varphi(\zeta,\omega,y)
\dd\!\pr{(f_\omega^{(\zeta_0)})_*
\widehat\mu_\omega}\!(y)
\dd\widehat\nu_N(\zeta,\omega)
\end{equation}
for every nonnegative measurable function $\varphi$.

Note also that $(\omega,x,r)\mapsto\widehat\mu_\omega(B(x,r))$ is jointly measurable, so the set $X_N^0$ of $z\in X_N$ with $x_k\in\supp\widehat\mu_{\omega_k}$ for every $k\in\Z$ is measurable, $F_N$-invariant, and of full $\mu_N$-measure. With this in mind, as before for $z\in X_N^0$ and $0<r\le r_0$, we set
\[
\psi_r(z):=\log\frac{\widehat\mu_{\omega_{-1}}\pr{f_{-1}^{-1}B(x_0,r)}}{\widehat\mu_{\omega_0}\pr{B(x_0,r)}},
\]
and set $\psi_r:=0$ on $X_N\setminus X_N^0$. By Lemma~\ref{lem:blocking}, the measures $(f_\omega^{(\zeta_0)})_*\widehat\mu_\omega$ and $\widehat\mu_{\widehat\theta^{\zeta_0}\omega}$ have finite relative entropy for $\widehat\nu_N$-almost every $(\zeta,\omega)$. Hence Lemma~\ref{lem:ball-quotient} and \eqref{eq:general-integral} show, as in Corollary~\ref{cor:psi}, that $\psi_r$ converges $\mu_N$-almost everywhere as $r\to0$ to a measurable function $\psi$. Then again setting $\Psi:=2\sup_{q\in\Q\cap(0,r_0]}|\psi_q|$, we obtain $|\psi_r|+|\psi|\le\Psi$ for $0<r\le r_0$, and
\[
\int_{X_N}\psi\,\dd\mu_N=h^{\mathrm F}(\widehat\mu_N),
\qquad
\int_{X_N}\Psi\,\dd\mu_N\le2C\pr{1+h^{\mathrm F}(\widehat\mu_N)}<\infty.
\]

To prove (i), assume $\lambda_{\mathrm{top}}<0$ and let $\ep\in(0,-\lambda_{\mathrm{top}}/4)$. Then $A^+<0$ by \eqref{eq:general-scales}. We define the radii $r_{n,j}^+$ analogously as in Section~\ref{twoside} using the $f_k$ and $x_k$ above. The proof of Lemma~\ref{lem:telescope} and telescoping applies almost exactly as before with $\mu(B(x_k,r))$ replaced by $\widehat\mu_{\omega_k}(B(x_k,r))$. The proofs of Lemmas~\ref{lem:maker} and~\ref{lem:maker-average} also apply with $(X,\mu_\eta,F)$ replaced by $(X_N,\mu_N,F_N)$ and $h_\mu^{\mathrm F}(\eta)$ by $h^{\mathrm F}(\widehat\mu_N)$.

In Lemma~\ref{lem:endpoint-mass}, the initial ball is now measured by $\widehat\mu_{\omega_{-n}}$, which depends on $n$. Since the constant $C_0$ in \eqref{eq:packing} depends only on $M$, applying \eqref{eq:packing} to each $\widehat\mu_\omega$ and integrating over $(\zeta,\omega)$ gives us
\[
\widehat\mu_N\pr{\set*{(\zeta,\omega,x)\st\widehat\mu_\omega(B(x,r))<q}}\le C_0qr^{-d_M}
\qquad(0<r\le r_0,\ q>0).
\]
Since $(\pi\circ F_N^{-n})_*\mu_N=\widehat\mu_N$, the Borel--Cantelli argument in the proof of Lemma~\ref{lem:endpoint-mass}, with $\mu(B(g_n(y),s_n(y)))$ replaced by $\widehat\mu_{\omega_{-n}}(B(x_{-n},r_{n,0}^+))$ and the above bound in place of \eqref{eq:packing}, shows that $\frac1n\log\widehat\mu_{\omega_{-n}}\pr{B(x_{-n},r_{n,0}^+)}\to0$ for $\mu_N$-almost every $z$.

Finally, the terminal measure $\widehat\mu_{\omega_0}$ does not depend on $n$, so Lemma~\ref{lem:interpolation} applies to $\widehat\mu_{\omega_0}$ with the radii $r_{n,n}^+$. Arguing as in Section~\ref{sec:proof}, using \eqref{eq:general-scales} and $h^{\mathrm F}(\widehat\mu_N)=(N+\frac12)h^{\mathrm F}(\widehat\mu)$, we obtain for $\mu_N$-almost every $z$ that
\begin{align*}
\overline d_{\widehat\mu_{\omega_0}}(x_0)&\le\frac{h^{\mathrm F}(\widehat\mu)}{-\lambda_{\mathrm{top}}-2\ep}.
\end{align*}
The maps $(\omega,x)\mapsto\underline d_{\widehat\mu_\omega}(x)$ and $(\omega,x)\mapsto\overline d_{\widehat\mu_\omega}(x)$ are measurable, since the limits defining them may be computed along $r=2^{-k}$, and the image of $\mu_N$ under $z\mapsto(\omega_0,x_0)$ is $\widehat\mu$. Hence these bounds hold for $\widehat\mu$-almost every $(\omega,x)$. Letting $\ep\to0$ along a sequence and disintegrating $\widehat\mu$ over $\omega$ proves (i). The proof of (ii) is entirely analogous by replacing $\lambda_{\mathrm{top}}$, $a^+$, $A^+$, and $r_{n, j}^+$ in the above with $\lambda_{\mathrm{bot}}$, $a^-$, $A^-$, and $r_{n, j}^-$ respectively. Combining (i) and (ii) gives the final part of Theorem~\ref{thm:general-base}.
\qed


\begin{thebibliography}{99}
		
		\bibitem{AvilaViana2010}
		A.~Avila and M.~Viana,
		\emph{Extremal Lyapunov exponents: an invariance principle and applications},
		Invent. Math. \textbf{181} (2010), no.~1, 115--178.
		
		
		\bibitem{BarreiraPesinSchmeling1999}
		L.~Barreira, Ya.~Pesin, and J.~Schmeling,
		\emph{Dimension and product structure of hyperbolic measures},
		Ann. of Math. (2) \textbf{149} (1999), no.~3, 755--783.
		
		\bibitem{Baxendale1989}
		P.~H.~Baxendale,
		\emph{Lyapunov exponents and relative entropy for a stochastic flow of
			diffeomorphisms},
		Probab. Theory Related Fields \textbf{81} (1989), 521--554.
		
		\bibitem{BezerraSalcedo2025}
		J.~Bezerra and G.~Salcedo,
		\emph{A version of Oseledets for proximal random dynamical systems on the circle},
		preprint, arXiv:2503.16026, 2025.
		
		\bibitem{BrownRodriguezHertz2026}
A.~Brown and F.~Rodriguez Hertz,
\emph{Entropy formula for Furstenberg entropy: invariance principle and
exact dimensionality of contracting stationary measures},
unpublished manuscript, version dated September~2, 2026, 78 pp.
		
		\bibitem{Crauel90}
		H.~Crauel,
		\emph{Extremal exponents of random dynamical systems do not vanish},
		J. Dynam. Differential Equations \textbf{2} (1990), no.~3, 245--291.
		
		\bibitem{Crauel93}
		H.~Crauel,
		\emph{Non-Markovian invariant measures are hyperbolic},
		Stochastic Process. Appl. \textbf{45} (1993), no.~1, 13--28.
		
		\bibitem{EskinLindenstrauss}
		A.~Eskin and E.~Lindenstrauss,
		\emph{Random walks on locally homogeneous spaces},
		preprint, 2018.
 
\bibitem{FalconerTechniques}
K.~J. Falconer,
\emph{Techniques in Fractal Geometry},
John Wiley \& Sons, Chichester, 1997.
        
		\bibitem{Federer}
		H.~Federer,
		\emph{Geometric Measure Theory},
		Grundlehren Math. Wiss., vol.~153,
		Springer-Verlag, New York, 1969.

        \bibitem{Feng2023}
D.-J. Feng,
\emph{Dimension of invariant measures for affine iterated function systems},
Duke Math. J. \textbf{172} (2023), no.~4, 701--774.
        
		\bibitem{FengHu2009}
		D.-J. Feng and H. Hu,
		\emph{Dimension theory of iterated function systems},
		Comm. Pure Appl. Math. \textbf{62} (2009), no.~11, 1435--1500.

        
        
		\bibitem{HeJiaoXu2023}
		W.~He, Y.~Jiao, and D.~Xu,
		\emph{On the dimension theory of random walks and group actions by circle
			diffeomorphisms},
		preprint, arXiv:2304.08372, 2023.
		
		\bibitem{HochmanSolomyak2017}
		M.~Hochman and B.~Solomyak,
		\emph{On the dimension of Furstenberg measure for
			$\mathrm{SL}_2(\mathbb R)$ random matrix products},
		Invent. Math. \textbf{210} (2017), no.~3, 815--875.
		
		\bibitem{Kingman}
		J.~F.~C. Kingman,
		\emph{The ergodic theory of subadditive stochastic processes},
		J. Roy. Statist. Soc. Ser. B \textbf{30} (1968), 499--510.
		
		\bibitem{Ledrappier}
		F.~Ledrappier,
		\emph{Quelques propri\'et\'es des exposants caract\'eristiques},
		in \emph{\'Ecole d'\'Et\'e de Probabilit\'es de Saint-Flour XII--1982},
		Lecture Notes in Math., vol.~1097, 1984, 305--396.
		
		\bibitem{LedrappierPositivity}
		F.~Ledrappier,
		\emph{Positivity of the exponent for stationary sequences of matrices},
		in \emph{Lyapunov Exponents (Bremen, 1984)}, Lecture Notes in Math., vol.~1186, 1986, 56--73.
		
		\bibitem{LedrappierLessa2023}
		F.~Ledrappier and P.~Lessa,
		\emph{Exact dimension of Furstenberg measures},
		Geom. Funct. Anal. \textbf{33} (2023), 245--298.
		
		\bibitem{LL}
		F.~Ledrappier and P.~Lessa,
		\emph{Exact dimension of dynamical stationary measures},
		J. Mod. Dyn. \textbf{20} (2024), 679--715.
		
		\bibitem{LedrappierYoung1985}
		F.~Ledrappier and L.-S.~Young,
		\emph{The metric entropy of diffeomorphisms. Part II. Relations between
			entropy, exponents and dimension},
		Ann. of Math. (2) \textbf{122} (1985), no.~3, 540--574.
		
		\bibitem{LedrappierYoung1988}
		F.~Ledrappier and L.-S.~Young,
		\emph{Dimension formula for random transformations},
		Comm. Math. Phys. \textbf{117} (1988), no.~4, 529--548.
		
		\bibitem{Lessa}
		P.~Lessa,
		\emph{Entropy and dimension of disintegrations of stationary measures},
		Trans. Amer. Math. Soc. Ser. B \textbf{8} (2021), no.~4, 105--129.
		
		\bibitem{LiuXie2006}
		P.-D.~Liu and J.-S.~Xie,
		\emph{Dimension of hyperbolic measures of random diffeomorphisms},
		Trans. Amer. Math. Soc. \textbf{358} (2006), no.~9, 3751--3780.
		
		\bibitem{Maker}
		P.~T.~Maker,
		\emph{The ergodic theorem for a sequence of functions},
		Duke Math. J. \textbf{6} (1940), no.~1, 27--30.

           \bibitem{MihailescuUrbanski2016}
   E.~Mihailescu and M.~Urba\'nski,
   \emph{Random countable iterated function systems with overlaps and applications},
   Adv. Math. \textbf{298} (2016), 726--758.
        
		\bibitem{Petersen}
		P.~Petersen,
		\emph{Riemannian Geometry},
		3rd ed., Grad. Texts in Math., vol.~171,
		Springer, Cham, 2016.
		
		\bibitem{Rapaport2021}
		A.~Rapaport,
		\emph{Exact dimensionality and Ledrappier--Young formula for the
			Furstenberg measure},
		Trans. Amer. Math. Soc. \textbf{374} (2021), no.~7, 5225--5268.
		
		\bibitem{Viana}
		M.~Viana,
		\emph{Lectures on Lyapunov Exponents},
		Cambridge Stud. Adv. Math., vol.~145,
		Cambridge Univ. Press, Cambridge, 2014.
		
		\bibitem{Young1982}
		L.-S.~Young,
		\emph{Dimension, entropy and Lyapunov exponents},
		Ergodic Theory Dynam. Systems \textbf{2} (1982), no.~1, 109--124.
		
	\end{thebibliography}
\end{document}